\documentclass[13pt]{amsart}

\usepackage{graphicx}
\usepackage{float}
\usepackage[nice]{nicefrac}
\usepackage{amssymb}
\usepackage{amsmath,amsthm,amscd,amssymb, mathrsfs}

\usepackage{latexsym}

\numberwithin{equation}{section}
\numberwithin{figure}{section}

\theoremstyle{plain}
\newtheorem{theorem}{Theorem}[section]
\newtheorem{lemma}[theorem]{Lemma}

\newtheorem{proposition}[theorem]{Proposition}
\newtheorem{conjecture}[theorem]{Conjecture}

\theoremstyle{definition}
\newtheorem{definition}[theorem]{Definition}

\theoremstyle{remark}
\newtheorem{remark}[theorem]{Remark}

\newtheorem{case[theorem]}{Case}

\def\({\left(}
\def\){\right)}
\def\[{\left[}
\def\]{\right]}
\def\<{\langle}
\def\>{\rangle}

\title[Weak version of restriction estimates ]{Weak version of restriction estimates for spheres and paraboloids in finite fields }

\author{Hunseok Kang and Doowon Koh}

\address{
Department of Mathematics\\
Soongsil University \\
Seoul, 156-743 Korea}
\email{hkang@ssu.ac.kr}

\address{
Department of Mathematics\\
Chungbuk National University \\
Cheongju city, Chungbuk-Do 361-763 Korea}
\email{koh131@chungbuk.ac.kr}

\keywords{ homogeneous function, finite field, restriction operators }
\thanks{This research was supported by Basic Science Research Program through the National
Research Foundation of Korea funded by the Ministry of Education, Science and Technology(2012R1A1A1014924, ~ 2012R1A1A1001510).}

\subjclass[2010]{Primary: 42B05; Secondary 43A32, 43A15 }

\begin{document}

\begin{abstract} We study $L^p-L^r$ restriction estimates for algebraic varieties in $d$-dimensional
vector spaces over finite fields. Unlike the Euclidean case, if the dimension $d$ is even, then it is
conjectured that the  $L^{(2d+2)/(d+3)}-L^2$ Stein-Tomas restriction result can be improved to the
$L^{(2d+4)/(d+4)}-L^2$ estimate for both  spheres and paraboloids in finite fields.
In this paper we show that the conjectured $L^p-L^2$ restriction estimate holds in the specific case when
test functions under consideration are  restricted to $d$-coordinate functions or homogeneous functions of
degree zero. To deduce our result, we use the connection between the restriction phenomena for our
varieties in $d$ dimensions and those for homogeneous varieties in $(d+1)$ dimensions.
\end{abstract}

\maketitle

\section{Introduction}

Let $V$ be a subset of ${\mathbb R}^d, d\geq 2,$ and $d\sigma$ a positive measure
supported on $V$. The classical restriction problem asks us to determine $1\leq p,r \leq \infty$ such that
the following restriction estimate holds:
\begin{equation}\label{defr}
  \|\widehat{f}\|_{L^r(V, d\sigma)}\le C_{p,r,d} ~\|f\|_{L^p(\mathbb R^d)}
\end{equation}
for every Schwarz function $f:\mathbb R^d \to \mathbb C.$
By duality, the restriction estimate (\ref{defr}) is same as the following extension estimate:
$$ \|(gd\sigma)^\vee\|_{L^{p^\prime}(\mathbb R^d)} \le C_{p,r,d} ~\|g\|_{L^{r^\prime}(V, d\sigma)},$$
where $p^{\prime}=p/(p-1)$ and $r^{\prime}=r/(r-1).$
This problem was addressed and studied by E.M. Stein (\cite{St78}).
Much attention has been given to this problem, in part because it is closely related to other important problems such as the Falconer distance problem, the Kakeya problem, and the Bochner-Riestz problem
(for example, see \cite{Er05, BCT06, Ca92, Ta99}). The complete answer to the restriction problem is known only for certain lower dimensional hypersurfaces.
For instance, Zygmund (\cite{Zy74}) established the restriction conjecture for the circle and the parabola in the plane.
Barcelo (\cite{Ba85}) and Wolff (\cite{Wo01}) also solved it for the cone of $\mathbb R^3$ and $\mathbb R^4$, respectively.
However,  the restriction conjecture remains open in other higher dimensions.
The best known result for the cone of $\mathbb R^d, d\geq 5$, is due to Wolff (\cite{Wo01}) who utilized the bilinear restriction method.
Terence Tao (\cite{Ta03}) also used the method to derive the best known restriction results on the sphere and paraboloid of $\mathbb R^d, d\geq 3.$
However, it has been believed that classically used analytical approaches are not enough to settle down the restriction problem.
We refer reader to Tao's survey paper \cite{Ta04} and references therein for currently known skills to deduce restriction results in the Euclidean case. \\

In recent years, problems in the Euclidean space have been studied in the finite field setting.
Motivation on the study of  Euclidean problems in finite fields is to understand the original problems in simple finite field structure.
In 1999,  Tom Wolf (\cite{Wo99}) formulated the Kakeya problem in finite fields and
new results on the problem were addressed  in the subsequent papers (see \cite{Ro01, MT04, Ta05}).
Surprisingly, Dvir (\cite{Dv09}) proved the finite field Kakeya conjecture by  beautifully simple, new argument based on the polynomial method.
His work has inspired researchers to further efforts for seeking solutions to other analysis problems in finite fields.
In \cite{MT04}, Mockenhaupt and Tao first investigated the Fourier restriction problem for various algebraic varieties in the finite field setting
and they addressed interesting results on this problem. Further efforts to understand the finite field restriction problem have been made by other researchers (see, for example, \cite{IK09, IK10, Ko13, KS12, KS1, LL13, LL10}). In particular,  the finite field restriction problem for cones, paraboloids, and spheres have been mainly studied, but known results are far from the conjectured results in higher dimensions.  \\

When we study analogue of Euclidean problems in finite fields, we often find an unprecedented phenomenon
which never occurs in the Euclidean case.
 It is well known that if $V\subset \mathbb R^d$ is the sphere or a compact subset of the paraboloid,
then  $ p_0=(2d+2)/(d+3)$ gives the sharp $p $ exponent for  $L^p-L^2 $ restriction estimates for $V.$
The number $p_0$ is called the Stein-Tomas exponent for the $L^p-L^2$ restriction inequality.
On the contrary to the Euclidean case,   it is possible to improve the Stein-Tomas exponent $p_0$  if $V$ is the paraboloid in {\bf even} dimensional vector spaces over finite fields. For example, Mockenhaupt and Tao (\cite{MT04}) proved
the $L^{4/3}-L^2$ restriction estimate  for the parabola lying in two dimensional vector spaces over finite fields.
For even dimensions $d\geq 4$,  A. Lewko and M. Lewko (\cite{LL10}) obtained  the $L^{2d^2/(d^2+2d-2)}-L^2$ restriction result for the paraboloid in the  finite field setting. These results are clearly better than the Stein-Tomas inequality. Here, we point out that if the dimension $d\geq 3$ is {\bf odd} and $-1$ is a square number, then it is impossible to improve the Stein-Tomas restriction estimate for spheres or paraboloids in finite field case. For this reason, we shall just focus on studying  the $L^p-L^2$ restriction estimates for spheres or paraboloids in {\bf even} dimensions.  \\

 When $-1$ is a square number in the underlying finite field,  it is conjectured that the $L^{(2d+4)/(d+4)}-L^2$ restriction estimate is the best possible result on the $L^p-L^2$ estimate for the sphere or the paraboloid in {\bf even} dimensional vector spaces over finite fields (see Conjecture \ref{conjecture1}).
The conjecture is open except for $d=2,$ and  the aforemensioned result due to A. Lewko and M. Lewko is far from the conjectured one. Furthermore,  there is no known result for spheres in even dimensions $d\geq 4$ which improves on the Stein-Tomas exponent.
The main purpose of this paper is to find  a class of test functions for which the conjectured $L^{(2d+4)/(d+4)}-L^2 $ restriction estimate holds for the sphere or the paraboloid in even dimensional vector spaces over finite fields.  The main idea to derive our results is to use a connection between restriction estimates for homogeneous varieties in $(d+1)$ dimensions and those for the sphere or the paraboloid in $d$-dimensional vector spaces over finite fields.
\section{ Weak version of  restriction problems}
To precisely state our main results, we shall introduce the weak version of restriction problems in the finite field setting. Roughly speaking,  we investigate the $L^p-L^2$ restriction estimates for algebraic varieties in the specific case when the test functions are restricted to specific classes of functions rather than all functions on vector spaces over finite fields. We begin by reviewing the restriction problem for algebraic varieties in finite fields.
\subsection{Review of the restriction problem}
Let  $\mathbb F_q^d, d\geq 2,$ be the $d$-dimensional vector spaces over finite fields  $\mathbb F_q$ with $q$ elements.  We assume that the characteristic of $\mathbb F_q$ is greater than two.
The space $\mathbb F_q^d$ is equipped with a counting measure $dm$, by setting, for any function $g: (\mathbb F_q^d, dm) \to \mathbb C,$
$$\int_{\mathbb F_q^d} g(m)~dm = \sum_{m\in \mathbb F_q^d} g(m).$$
Here and throughout the paper, we write the notation $(\mathbb F_q^d, dm)$ for the space $\mathbb F_q^d$ with the counting measure $dm.$
On the contrary to the space $(\mathbb F_q^d, dm)$,  we endow its dual space with a normalized counting measure $dx.$
The dual space of $(\mathbb F_q^d, dm)$  is denoted by the notation $(\mathbb F_q^d, dx).$
Recall that if $g: (\mathbb F_q^d, dm) \to \mathbb C,$ then its Fourier transform $\widehat{g}$ is a function on the dual space $(\mathbb F_q^d, dx).$ Thus,  for $x\in (\mathbb F_q^d, dx)$,
$$\widehat{g}(x)=\int_{\mathbb F_q^d} \chi(-m\cdot x) g(m)~ dm=\sum_{m\in \mathbb F_q^d} \chi(-m\cdot x) g(m),$$
where $\chi$ denotes a nontrivial additive character of $\mathbb F_q.$
Also recall that if $f: (\mathbb F_q^d, dx) \to \mathbb C$, then its inverse Fourier transform $f^\vee$ can be defined by
$$ f^\vee(m)=\int_{\mathbb F_q^d} \chi( m\cdot x) f(x)~ dx= \frac{1}{q^d} \sum_{x\in \mathbb F_q^d} \chi(m\cdot x) f(x)$$
where $m\in (\mathbb F_q^d, dm).$  Using the orthogonality relation of $\chi$, one can easily show that $(\widehat{g})^\vee(m)=g(m)$ for $g:(\mathbb F_q^d, dm)\to \mathbb C.$ This provides us of the Fourier inversion theorem:
\begin{equation}\label{FI} g(m)=\int_{\mathbb  F_q^d} \chi(m\cdot x) \widehat{g}(x)~dx =\frac{1}{q^d} \sum_{x\in \mathbb F_q^d} \chi(m\cdot x) \widehat{g}(x).\end{equation}

 Let $V$ be an algebraic variety in the dual space $(\mathbb F_q^d, dx).$ The variety $V$ is equipped with the normalized surface measure $d\sigma$, which is defined by the relation
$$ \int f(x) ~d\sigma(x) = \frac{1}{|V|}\sum_{x\in V} f(x),$$
where $f:(\mathbb F_q^d, dx) \to \mathbb C.$
Observe that we can write $d\sigma(x)= \frac{q^d}{|V|} V(x)~dx.$ Here, and throughout this paper, we write $A(x)$ for the characteristic function on a set $A\subset \mathbb F_q^d$ and $|A|$ denotes the cardinality of the set $A.$\\

The restriction problem for the variety $V$ is to determine $1\leq p, r\leq \infty$ such that
the following restriction estimate holds:
\begin{equation}\label{restriction} \|\widehat{g}\|_{L^r(V, d\sigma)} \leq C \|g\|_{L^p(\mathbb F_q^d, dm)} \quad\mbox{for all functions}~~g:\mathbb F_q^d \to \mathbb C,\end{equation}
where the constant $C>0$ is independent of functions $g$ and the size of the underlying finite field $\mathbb F_q.$ The notation $R(p\to r)\lesssim 1$ is used to indicate that the restriction inequality (\ref{restriction}) holds. In this case, we say that  the $L^p-L^r$ restriction estimate holds.
 By duality,  inequality (\ref{restriction}) is same as the following extension estimate:
\begin{equation}\label{extension} \|(gd\sigma)^\vee\|_{L^{p^\prime}(\mathbb F_q^d, dm)} \leq C \|g\|_{L^{r^\prime}(V, d\sigma)}.\end{equation}
When this extension inequality holds, we say that the $L^{r^\prime}-L^{p^\prime}$ extension estimate holds and  we write  $R^*(r^\prime\to p^\prime)\lesssim 1$ for it.  Thus, $R(p\to r)\lesssim 1 $ if and only if $R^*(r^\prime \to p^\prime)\lesssim 1.$
\begin{remark}  $A\lesssim B$ for $A,B>0$ means that there exists $C>0$ independent of $q=|\mathbb F_q|$ such that $A\le CB.$ We also write $B\gtrsim A$ for $A\lesssim B.$ In addition, $A\sim B$ means that $A\lesssim B$ and $A\gtrsim B.$ We can define $R(p\to r)$ to be the best constant such that the restriction estimate (\ref{restriction}) holds. $R(p\to r)$ may depend on $q$. The restriction problem is to determine $p,r$ such that $R(p\to r)\lesssim 1.$
\end{remark}

When $V\subset \mathbb F_q^d$ is the sphere or the paraboloid, the necessary conditions for $R(p\to r)\lesssim 1$ are well known.
In particular,  necessary conditions for $R(p\to 2)\lesssim 1$  mainly depend on the biggest size of the affine subspaces lying in the variety $V.$
For example,   if $-1\in \mathbb F_q$ is a square number and $V\subset \mathbb F_q^d$ is the sphere or the paraboloid, then one can construct an affine subspace $H \subset \mathbb F_q^d$ such that
$|H|=q^{(d-1)/2}$ for $d\geq 3$ odd and $|H|=q^{(d-2)/2}$ for $d\geq 2$ even (see \cite{IK09} and \cite{IR07}). Taking $g(x)=H(x)$ in (\ref{extension}), we can directly deduce that the necessary conditions for $R(p\to 2)\lesssim 1$ are given by
\begin{equation}\label{oddN} 1\leq p\leq \frac{2d+2}{d+3}\quad \mbox{for odd}~~ d\ge 3\end{equation}
and
\begin{equation}\label{evenN} 1\leq p\leq \frac{2d+4}{d+4}\quad \mbox{for even}~~ d\ge 2.\end{equation}
It was proved in \cite{MT04} and \cite{IK08} that the Stein-Tomas inequality  holds for the sphere and the paraboloid, respectively.
Therefore,  if $d\ge 3$ is {\bf odd}, then (\ref{oddN}) is also the sufficient condition for $R(p\to 2)\lesssim 1.$
However, when the dimension $d$ is {\bf even}, it is not known that (\ref{evenN}) is the sufficient condition for $R(p\to 2)\lesssim 1$ except for dimension two.
For this reason, by the nesting property of norms,  one may want to establish  the following conjecture.
\begin{conjecture}\label{conjecture1} Let $V\subset \mathbb F_q^d$ be the sphere or the paraboloid.  Assuming that $-1\in \mathbb F_q$ is a square number and $d\geq 4$ is even, then
$$R\left(\frac{2d+4}{d+4}\to 2\right)\lesssim 1.$$
\end{conjecture}

\subsection{$d$-coordinate lay functions and homogeneous functions of degree zero}
We introduce specific test functions  on which  the restriction operator for the sphere or the paraboloid acts.
The following two definitions are closely related to  a weak version of the restriction problem for the paraboloid.
\begin{definition}
A function $g:(\mathbb F_q^d, dm)\to \mathbb C$ is called a $d$-coordinate lay function if it satisfies that for each $(m^\prime, m_d)\in \mathbb F_q^{d-1}\times \mathbb F_q,$
$$ g(m^\prime, m_d)=g(m^\prime, s m_d) \quad\mbox{for all} ~~s\in \mathbb F_q\setminus\{0\}.$$\end{definition}

\begin{definition} We write $R_{d-lay}(p\to r)\lesssim 1$ if the restriction estimate (\ref{restriction}) holds for all $d$-coordinate lay functions $g:(\mathbb F_q^d, dm) \to \mathbb C.$
\end{definition}

The weak version of the restriction operator for the sphere shall be defined by taking homogeneous  functions of degree zero as test functions.
As usual,  a function $g:(\mathbb F_q^d, dm)\to \mathbb C$ is named a homogeneous function of degree zero if $ g(s m)=g(m)$ for $m \in \mathbb F_q^d, ~s \in \mathbb F_q\setminus\{0\}.$
\begin{definition} We write $R_{hom}(p\to r)\lesssim 1$ if the restriction estimate (\ref{restriction}) holds for all homogeneous functions of degree zero, $g:(\mathbb F_q^d, dm) \to \mathbb C.$
\end{definition}

\subsection{Statement of main results}  Our first result below is related to the parabolical restriction estimate for $d-lay$ test functions.
\begin{theorem}\label{main3} Let $d\sigma$ be the normalized surface measure on the paraboloid
$P:=\{x\in \mathbb F_{q}^d: x_1^2+\cdots+x_{d-1}^2=x_d\}.$ If  $d\geq 2$ is even, then we have
$$ R_{d-lay}\left(\frac{2d+4}{d+4}\to 2\right)\lesssim 1.$$
\end{theorem}

When the test functions are homogeneous functions of degree zero, we obtain the strong result on the weak version of spherical restriction problems.
\begin{theorem}\label{main4} Let $d\sigma$ be the normalized surface measure on the sphere with nonzero radius
$S_j:=\{x\in \mathbb F_{q}^d: x_1^2+\cdots+x_{d}^2=j\neq 0\}.$  Then if $d\geq 2$ is even, we have
$$ R_{hom}\left(\frac{2d+4}{d+4}\to 2\right)\lesssim 1.$$
\end{theorem}
Conjecture \ref{conjecture1} claims that  if $d\geq 4$ is even, then  $(2d+4)/(d+4)$ is the optimal $p$ value for the $L^p-L^2$ restriction estimate for spheres and paraboloids in finite fields.
According to  Theorem \ref{main3} and \ref{main4}, it seems that the conjecture is true.
In dimension two, this conjecture  was actually proved by Mockenhaupt and Tao (\cite{MT04}) for the parabola and Iosevich and Koh (\cite{IK08}) for the circle. Indeed, they obtained the $L^2-L^4$ extension estimate which exactly implies that $L^{4/3}-L^2$ restriction estimate holds.
However, it is still open in higher even dimensions $d\geq 4$ and the currently best known result for the paraboloid is $R(2d^2/(d^2+2d-2)\to 2)\lesssim 1$ due to A. Lewko and M. Lewko (\cite{LL10}).
In fact, they proved the extension estimate, $R^*(2\to 2d^2/(d^2-2d+2))\lesssim 1$ for even $d\geq 4.$
Notice that this result is much better than the Stein-Tamas inequality, that is $R( (2d+2)/(d+3) \to 2)\lesssim 1.$ For the sphere in even dimensions $d\geq 4,$  the Stein-Tomas inequality was only  obtained by Iosevich and Koh (\cite{IK08}) and it has not been improved.
\subsection{Outline of the remain parts of the paper}
The remain parts of this paper are constructed for providing  proofs of  Theorem \ref{main3} and \ref{main4}. In Section \ref{conerestriction}, we deduce the $L^p-L^2$ restriction estimate for homogeneous varieties in $d+1$ dimensional vector spaces over finite fields $\mathbb F_q$.
Since homogeneous varieties  are a collection of lines, it sounds plausible to expect that the Fourier decay of them is not so good. However, it is not always true. Indeed, we observe that if $(d+1)$ is odd, then the Fourier decay of homogeneous varieties in (d+1) dimensions is enough to derive a good $L^p-L^2$ restriction result from the Stein-Tomas argument. In Section \ref{connection},  we complete the proofs of Theorem \ref{main3} and \ref{main4} by deducing the connection between a weak version of restriction estimates for spheres or paraboloids in $d$ dimensions and  the restriction estimates for homogeneous varieties in $d+1$ dimensions.

\section{ Restriction phenomenon for homogeneous varieties}\label{conerestriction}
Let $d\geq 2$ be an integer. In this section, we derive the $L^p-L^2$ estimate for homogeneous varieties lying in $(\mathbb F_q^{d+1}, d\overline{x})$ where $d\overline{x}$ denotes the normalized counting measure on $\mathbb F_q^{d+1}.$ Define a variety $C\subset (\mathbb F_q^{d+1}, d\overline{x})$ as
$$ C=\{(x, x_{d+1})\in \mathbb F_q^d\times \mathbb F_q: x_1^2+\cdots+x_{d-1}^2=x_d x_{d+1} \}.$$
For each $j\in \mathbb F_q^*,$  define a variety $H_j\subset (\mathbb F_q^{d+1}, d\overline{x})$ by
 $$H_j=\{(x, x_{d+1})\in \mathbb F_q^d\times \mathbb F_q: x_1^2+\cdots+x_d^2=j x_{d+1}^2 \}.$$
Throughout this paper, we denote by $d\sigma_c$ and $d\sigma_j$  the normalized surface measures on $C$ and $H_j$, respectively. In addition,  $(\mathbb F_q^d, d\overline{m})$ denotes the dual space of $ (\mathbb F_q^{d+1}, d\overline{x})$ where $d\overline{m}$ is the counting measure on $\mathbb F_q^{d+1}.$  Recall that if $\overline{m} \in (\mathbb F_q^{d+1}, d\overline{m})$,  then
$$ (d\sigma_c)^\vee(\overline{m})=\int_C \chi(\overline{m}\cdot \overline{x})~ d\sigma_c(\overline{x})= \frac{1}{|C|} \sum_{\overline{x}\in C}  \chi(\overline{m}\cdot \overline{x})$$
and
$$ (d\sigma_j)^\vee(\overline{m})=\int_{H_j} \chi(\overline{m}\cdot \overline{x}) ~d\sigma_c(\overline{x})= \frac{1}{|H_j|} \sum_{\overline{x}\in H_j}  \chi(\overline{m}\cdot \overline{x}).$$

With the above notation, we have the following result.
\begin{lemma}\label{keyR} Let $d\geq 2$ be even. Then
$$|C|=|H_j|=q^d \quad \mbox{for}~~j\in \mathbb F_q^*.$$
Moreover, if $\overline{m}\in \mathbb F_q^{d+1}\setminus \{ (0,\cdots,0)\},$ then
 $$\left |(d\sigma_c)^\vee(\overline{m})\right| \leq  q^{-d/2}$$
and
$$ \left| (d\sigma_j)^\vee(\overline{m})\right|\leq  q^{-d/2} \quad \mbox{for all}~~ j\in \mathbb F_q^*.$$
\end{lemma}
\begin{proof}
Before we proceed with the proof, we recall  preliminary knowledge for exponential sums.
Let $\eta$ be a quadratic character of $\mathbb F_q.$
For each $a\in \mathbb F_q$,  the absolute value of the Gauss sum $G_a$  is given by
\begin{equation}\label{gauss} |G_a| :=\left| \sum_{s\in \mathbb F_q^*} \eta(s) \chi(as)\right|=\left| \sum_{s\in \mathbb F_q^*} \eta(s) \chi({a}/{s}) \right|= \left\{\begin{array}{ll} q^{\frac{1}{2}} \quad &\mbox{if} \quad a\ne 0\\
                                                  0 \quad & \mbox{if} \quad a=0. \end{array}\right.\end{equation}
It is not hard to see that
\begin{equation}\label{square} \sum_{s \in {\mathbb F}_q} \chi (a s^2) = G_1 \eta(a)  \quad \mbox{for any} \quad a \ne 0.\end{equation}
It follows from the orthogonality relations of $\chi$ and $\eta$ that
\begin{equation}\label{orthogonality} \sum_{s\in \mathbb F_q} \chi(as)=\left\{\begin{array}{ll} 0 &\quad\mbox{if}~~a\in \mathbb F_q^*\\
                                                                                 q &\quad\mbox{if}~~a=0, \end{array} \right.\end{equation}
and
$$\sum_{s\in \mathbb F_q^*} \eta(as)=\left\{\begin{array}{ll} 0 &\quad\mbox{if}~~a\in \mathbb F_q^*\\
                                                                                 q-1 &\quad\mbox{if}~~a=0. \end{array} \right.$$
For (\ref{gauss}), (\ref{square}), and (\ref{orthogonality}),  see Chapter 5 in \cite{LN97}.
Completing the square and using a change of variables,  (\ref{square}) can be generalized by the formula:
\begin{equation}\label{complete}  \sum_{s\in \mathbb F_q} \chi(as^2+bs) = G_1 \eta(a)  \chi({b^2}/{(-4a)})
\quad\mbox{for}~~a\in \mathbb F_q^*,  b\in \mathbb F_q.\end{equation}

Now we are ready to prove the lemma. First, we estimate $(d\sigma_c)^\vee.$
For $\overline{m}=(m_1,\cdots, m_{d+1})\in \mathbb F_q^{d+1},$  it follows from the orthogonality relation of $\chi$ that
\begin{align*}  (d\sigma_c)^\vee(\overline{m})=&\frac{1}{|C|} \sum_{\overline{x}\in C } \chi(\overline{m}\cdot\overline{x})\\
=& \frac{1}{q|C|} \sum_{t\in \mathbb F_q}\sum_{\overline{x}\in \mathbb F_q^{d+1}} \chi(\overline{m}\cdot\overline{x})\chi(t(x_1^2+\cdots+x_{d-1}^2-x_dx_{d+1}))\\
=& \frac{q^d}{|C|} \delta_0(\overline{m}) + \frac{1}{q|C|} \sum_{t\ne 0}\sum_{\overline{x}\in \mathbb F_q^{d+1}} \chi(\overline{m}\cdot\overline{x})\chi(t(x_1^2+\cdots+x_{d-1}^2-x_dx_{d+1}))\\
:=& \mbox{I} +\mbox{II}
\end{align*}
where $\delta_0(\overline{m})=1$ for $\overline{m}=(0,\dots,0)$, and $0$ otherwise.
Applying (\ref{complete}), we see that
$$\mbox{II}=\frac{G_1^{d-1}}{q|C|} \sum_{t\neq 0} \eta(t)^{d-1} \chi(\|m^\prime\|^2/(-4t)) \sum_{x_{d+1}\in \mathbb F_q} \chi(m_{d+1}x_{d+1})
\sum_{x_d\in \mathbb F_q} \chi((m_d-tx_{d+1})x_d),$$
where we define that $\|m^\prime\|^2=m_1^2+\cdots+ x_{d-1}^2.$
Since $d$ is even and $\eta$ is the quadratic character of $\mathbb F_q$,  we see $\eta(t)^{d-1} =\eta(t).$ In addition, notice from the orthogonality relation of $\chi$ that
$$\sum_{x_{d+1}\in \mathbb F_q} \chi(m_{d+1}x_{d+1})
\sum_{x_d\in \mathbb F_q} \chi((m_d-tx_{d+1})x_d) =q \chi(m_dm_{d+1}/t).$$
Then we obtain that for $\overline{m}\in \mathbb F_q^{d+1},$
\begin{equation}\label{form1} (d\sigma_c)^\vee(\overline{m})
=\frac{q^d}{|C|} \delta_0(\overline{m}) + \frac{G_1^{d-1}}{|C|} \sum_{t\neq 0} \eta(t) \chi((\|m^\prime\|^2-4m_dm_{d+1})/(-4t)).
\end{equation}
From the definition of $(d\sigma_c)^\vee$ and the orthogonality relation of $\eta$,  we see that
$$ 1= (d\sigma_c)^\vee(0,\cdots,0)= \frac{q^d}{|C|}.$$ Thus, we  completes the proof of
$|C|=q^d$ and it follows immediately from (\ref{form1}) and (\ref{gauss}) that
$|(d\sigma_c)^\vee(\overline{m})|\leq q^{-d/2}$ for $\overline{m}\ne (0,\dots, 0).$ \\
Next, we can directly deduce by the previous argument that if $j\in \mathbb F_q^*$ and $\overline{m}\in \mathbb F_q^{d+1}$, then
$$ (d\sigma_j)^\vee(\overline{m})
= \frac{q^d}{|H_j|} \delta_0(\overline{m}) + \frac{G_1^{d+1}}{q|H_j|} \eta(-j) \sum_{t\in \mathbb F_q^*} \eta(t) \chi( (m_{d+1}^2-j\|m\|^2)/(4jt)),$$
where $\|m\|^2=m_1^2+\cdots+ m_d^2.$ This implies that $|H_j|=q^d$ for $j\in \mathbb F_q^*$ and
$|(d\sigma_j)^\vee(\overline{m})|\leq q^{-d/2}$ for $\overline{m}\ne (0,\dots, 0).$ We leave the detail to readers.
\end{proof}
\begin{remark} If $d$ is odd, then the Fourier decays become much worse than those in conclusions of Lemma \ref{keyR}. To see this, notice that if $d$ is odd, then $\eta^{d-1}=1.$ Thus,  the term $\eta$ disappears in the formula (\ref{form1}). Consequently,  if $\overline{m}=(m^\prime, m_d, m_{d+1}) \neq (0,\dots,0)$ and $\|m^\prime\|^2-4m_dm_{d+1}=0,$ then $|(d\sigma_c)^\vee(\overline{m})|\sim q^{(-d+1)/2}.$
\end{remark}

Applying the well known Stein-Tomas argument in finite fields, Lemma \ref{keyR} enables us to deduce the $L^p-L^2$ restriction theorem for the homogeneous varieties $C$ and $H_j$ for $j\in \mathbb F_q^*.$
\begin{lemma}\label{Koh}  Let $d\geq 2$ be an even integer. Then we have
\begin{equation}\label{res1} \|\widehat{G}\|_{L^2(C, d\sigma_c)} \lesssim \|G\|_{L^{\frac{2d+4}{d+4}}(\mathbb F_q^{d+1}, d\overline{m})}\quad\mbox{for all functions}\quad G:\mathbb F_q^{d+1} \to \mathbb C.\end{equation}
We also have that if $j\in \mathbb F_q^*$, then
\begin{equation}\label{res2} \|\widehat{G}\|_{L^2(H_j, d\sigma_{j})} \lesssim \|G\|_{L^{\frac{2d+4}{d+4}}(\mathbb F_q^{d+1}, d\overline{m})}\quad\mbox{for all functions}~~ G:\mathbb F_q^{d+1} \to \mathbb C.\end{equation}
\end{lemma}

\begin{proof}
Since the proof of  (\ref{res1}) is exactly same as that of  (\ref{res2}),
we shall only introduce the proof of  (\ref{res1}).
By duality and H\"{o}lder's inequality, we see
\begin{align*} \|\widehat{G}\|^2_{L^2(C,d\sigma_c)} &=\sum_{\overline{m}\in \mathbb F_q^{d+1}} G(\overline{m})~ \overline{ (G\ast (d\sigma_c)^\vee)(\overline{m})}\\
&\leq\|G\|_{L^{\frac{2d+4}{d+4}}(\mathbb F_q^{d+1}, d\overline{m})} \|G\ast (d\sigma_c)^\vee \|_{L^{\frac{2d+4}{d}}(\mathbb F_q^{d+1}, d\overline{m})}.\end{align*}
It is enough to prove that for every function $G$ on $(\mathbb F_q^{d+1},d\overline{m}),$
$$\|G\ast (d\sigma_c)^\vee \|_{L^{\frac{2d+4}{d}}({\mathbb F_q^{d+1}}, d\overline{m})}\lesssim \|G\|_{L^{\frac{2d+4}{d+4}}({\mathbb F_q^{d+1}}, d\overline{m})}.$$

Define $ K(\overline{m})=(d\sigma_c)^\vee(\overline{m}) -\delta_0(\overline{m}).$ Since $(d\sigma)^\vee(0,\dots,0)=1$, we see that $ K(\overline{m})=0$ for $m=(0,\dots,0)$, and $K(\overline{m})=(d\sigma_c)^\vee(\overline{m})$ for $\overline{m}\in {\mathbb F_q^{d+1}}\setminus \{(0,\dots,0)\}.$ Since $G\ast (d\sigma_c)^\vee=G\ast \delta_0 + G\ast K,$ it will be enough to prove the following two inequalities:
\begin{equation}\label{E1} \|G\ast \delta_0 \|_{L^{\frac{2d+4}{d}}({\mathbb F_q^{d+1}}, d\overline{m})}\lesssim \|G\|_{L^{\frac{2d+4}{d+4}}({\mathbb F_q^{d+1}}, d\overline{m})}
\end{equation}
and
\begin{equation}\label{E2}\|G\ast K \|_{L^{\frac{2d+4}{d}}({\mathbb F_q^{d+1}}, d\overline{m})}\lesssim \|G\|_{L^{\frac{2d+4}{d+4}}({\mathbb F_q^{d+1}}, d\overline{m})}.
\end{equation}
Since $G\ast \delta_0(\overline{m})=G(\overline{m})$ for $\overline{m}\in (\mathbb F_q^{d+1}, d\overline{m}),$  (\ref{E1}) follows by observing that
$$\begin{array}{ll} \|G\ast \delta_0 \|_{L^{\frac{2d+4}{d}}({\mathbb F_q^{d+1}}, d\overline{m})}
&=\|G \|_{L^{\frac{2d+4}{d}}({\mathbb F_q^{d+1}}, d\overline{m})}\\
&\leq \|G\|_{L^{\frac{2d+4}{d+4}}({\mathbb F_q^{d+1}}, d\overline{m})},\end{array}$$
where the last line follows from the facts that $d\overline{m}$ is the counting measure and
$(2d+4)/(d+4) <(2d+4)/d.$ In order to prove (\ref{E2}),  we assume for a moment that
\begin{equation}\label{two} \|G\ast K\|_{L^2({\mathbb F_q^{d+1}}, d\overline{m})}\lesssim q \|G\|_{L^2(\mathbb F_q^{d+1},d\overline{m})}
\end{equation}
and
\begin{equation}\label{infty}\|G\ast K\|_{L^\infty({\mathbb F_q^{d+1}}, d\overline{m})}\lesssim q^{-\frac{d}{2}} \|G\|_{L^1(\mathbb F_q^{d+1},d\overline{m})}.
\end{equation}
Then (\ref{E2}) follows immediately by interpolating (\ref{two}) and (\ref{infty}). Thus, our final task is to show that  both (\ref{two}) and (\ref{infty}) hold.    As a direct consequence from the Plancherel theorem, (\ref{two}) can be proved. Indeed, we have
$$ \begin{array}{ll} \|G\ast K\|_{L^2(\mathbb F_q^{d+1}, d\overline{m})}&=\|\widehat{G}\widehat{K}\|_{L^2({\mathbb F_q^{d+1}}, d\overline{x})}\\
&\leq \|\widehat{K}\|_{L^\infty({\mathbb F_q^{d+1}}, d\overline{x})} \|\widehat{G}\|_{L^2({\mathbb F_q^{d+1}}, d\overline{x})}\\
&< q\|G\|_{L^2({\mathbb F_q^{d+1}}, d\overline{m})},\end{array}$$
where the last line is obtained by observing that for each $\overline{x}\in (\mathbb F_q^{d+1},d\overline{x})$\\
$|\widehat{K}(\overline{x})|= |d\sigma_c(\overline{x})-\widehat{\delta_0}(\overline{x})|=|q^{d+1}|C|^{-1} C(\overline{x})-1| < q.$
Now, we  prove (\ref{infty}). It follows from Young's inequality that
$$\|G\ast K\|_{L^\infty({\mathbb F_q^{d+1}}, d\overline{m})}\leq \| K\|_{L^\infty({\mathbb F_q^{d+1}}, d\overline{m})} \|G\|_{L^1({\mathbb F_q^{d+1}}, d\overline{m})}.$$
From the definition of $K$ and the Fourier decay estimate in Lemma \ref{keyR}, we conclude that
(\ref{infty}) holds. Thus, the proof is complete.

\end{proof}

\section{Proofs of Theorem \ref{main3} and \ref{main4}} \label{connection}
As a key ingredient of proving our main results, we use the relation between the restriction theorem for $C$ and $H_j$ in $\mathbb F_q^{d+1}$ and the weak restriction theorem for paraboloids and spheres in $\mathbb F_q^d.$ Theorem \ref{main3} shall be deduced from (\ref{res1}) in Lemma \ref{Koh}. Similarly, we shall prove Theorem \ref{main4} by applying (\ref{res2}) in Lemma \ref{Koh}.
\subsection{Proof of Theorem \ref{main3} }
We must prove that if $d\geq 2$ is even, then
$$\|\hat{g}\|_{L^2(P, d\sigma)} \lesssim \|g\|_{L^{(2d+4)/(d+4) }(\mathbb F_q^d, dm)} \quad \mbox{for all $d$-coordinate lay functions}\quad  g:\mathbb F_q^{d} \to \mathbb C.$$
Given a $d$-coordinate lay function $g: (\mathbb F_q^d, dm) \to \mathbb C,$  we define
$G_g: (\mathbb F_q^{d+1}, d\overline{m}) \to \mathbb C$ by the relation
\begin{equation}\label{DG} \widehat{G_g}(x^\prime, x_d, s)=\left\{\begin{array}{ll} \widehat{g}(x^\prime, x_d s)&\quad\mbox{if}~~s\ne 0\\
0&\quad\mbox{if}~~s= 0, \end{array} \right.\end{equation}
where $(x^\prime, x_d, s) \in (\mathbb F_q^{d+1}, d\overline{x})$ with $x^\prime\in \mathbb F_q^{d-1}, x_d, s\in \mathbb F_q.$
We need the  explicit form of $G_g.$
\begin{proposition}
For $(m, l)\in \mathbb F_q^d \times \mathbb F_q,$
$$ G_g(m, l) =\frac{g(m)}{q} \sum_{s\in \mathbb F_q^*} \chi(ls).$$
\end{proposition}\label{Pro}
\begin{proof} By the Fourier inversion theorem (\ref{FI}) for $d+1$ dimensions, and the definition of $\widehat{G_g}$ in (\ref{DG}),  we see that
if $(m^\prime, m_d, l) \in \mathbb F_q^{d-1}\times \mathbb F_q \times \mathbb F_q =\mathbb F_q^{d+1},$ then
\begin{align*} G_g(m^\prime, m_d, l)=&\frac{1}{q^{d+1}} \sum_{x^\prime\in \mathbb F_q^{d-1}, x_d, s\in \mathbb F_q} \chi(m^\prime \cdot x^\prime+ m_dx_d+ls) ~\widehat{G_g}(x^\prime, x_d, s)\\
=& \frac{1}{q^{d+1}} \sum_{s\neq 0} \sum_{(x^\prime, x_d)\in \mathbb F_q^d} \chi(m^\prime \cdot x^\prime+ m_dx_d+ls) ~\widehat{g}(x^\prime, x_d s).
\end{align*}
By a change of variables, $x_d\to x_d/s,$ and the Fourier inversion formula (\ref{FI}),
\begin{align*}G_g(m^\prime, m_d, l)=&\frac{1}{q^{d+1}} \sum_{s\ne 0} \chi(ls) \sum_{x\in \mathbb F_q^d} \chi(x\cdot (m^\prime, m_d/s))~ \widehat{g}(x)\\
=&\frac{1}{q} \sum_{s\ne 0} \chi(ls)~ g(m^\prime, m_d/s).\end{align*}
Since $g$ is a $d$-coordinate lay function,  $ g(m^\prime, m_d/s)=g(m)$ for all $s\in \mathbb F_q^*.$
Hence, the proof of Proposition \ref{Pro} is complete.\end{proof}
We continue to prove Theorem \ref{main3}. It is enough to show that
$$ \|\widehat{g}\|^2_{L^2(P, d\sigma)} \lesssim \|g\|^2_{L^{(2d+4)/(d+4)} (\mathbb F_q^d, dm )}.$$
Since $|C|=q^{d}=q|P|,$ it follows that
\begin{align*}\|\widehat{g}\|^2_{L^2(P, d\sigma)} =&\frac{1}{|P|}\sum_{x\in P} |\widehat{g}(x)|^2 \sim \frac{1}{|C|} \sum_{s\in \mathbb F_q^*} \sum_{x\in P} |\widehat{g}(x)|^2\\
=& \frac{1}{|C|} \sum_{s\in\mathbb F_q^*} \sum_{\substack{(x^\prime, x_d)\in \mathbb F_q^d\\
: x_1^2+\cdots+x_{d-1}^2=x_d s}} |\widehat{g}(x^\prime, x_d s)|^2\\
=&\|\widehat{G_g}\|^2_{L^2(C, d\sigma_c)}, \end{align*}
where the last line follows from (\ref{DG}).
By (\ref{res1}) in Lemma \ref{Koh},  to prove Theorem \ref{main3},  it therefore suffices to show that
$$ \|G_g\|^2_{L^{(2d+4)/(d+4)}(\mathbb F_q^{d+1}, d\overline{m})}\lesssim \|g\|^2_{L^{(2d+4)/(d+4)}(\mathbb F_q^{d}, dm)}.$$ Letting $\alpha=(2d+4)/(d+4) >1,$  it will be enough to prove that 
$$ \|G_g\|^{\alpha}_{L^{\alpha}(\mathbb F_q^{d+1}, d\overline{m})}\lesssim \|g\|^{\alpha}_{L^{\alpha}(\mathbb F_q^{d}, dm)}.$$
From the explicit form of  $G_g$  in Proposition \ref{Pro},  it follows that
\begin{align*}  \|G_g\|^{\alpha}_{L^{\alpha}(\mathbb F_q^{d+1}, d\overline{m})} =&\sum_{(m,l)\in \mathbb F_q^d \times \mathbb F_q} |G_g(m,l)|^{\alpha}\\
 =& \sum_{l\in \mathbb F_q} \sum_{m\in \mathbb F_q^d} |g(m)|^{\alpha} \left|q^{-1} \sum_{s\in \mathbb F_q^*} \chi(ls)\right|^{\alpha}\\
=& \sum_{m\in \mathbb F_q^d} |g(m)|^{\alpha} (q^{-1} (q-1) )^{\alpha} + \sum_{l\ne 0}
\sum_{m\in \mathbb F_q^d} q^{-\alpha} |g(m)|^{\alpha}\\
 \leq& \sum_{m\in \mathbb F_q^d} |g(m)|^{\alpha}+ q^{-\alpha} (q-1)\sum_{m\in \mathbb F_q^d} |g(m)|^{\alpha}  \leq 2 \|g\|^{\alpha}_{L^{\alpha}(\mathbb F_q^d, dm)}.
\end{align*}
Thus, the proof of Theorem \ref{main3} is complete.
\subsection{Proof of Theorem \ref{main4}}
We aim to prove that for every $j\in \mathbb F_q^*,$
$$\|\widehat{g}\|_{L^2(S_j, d\sigma)} \lesssim \|g\|_{L^{(2d+4)/(d+4)}(\mathbb F_q^d, dm)}$$
 for all homogeneous functions of degree zero  $g:\mathbb F_q^{d} \to \mathbb C.$
Let  $g:(\mathbb F_q^d, dm) \to \mathbb C$ be a homogeneous function of degree zero. By the definition of the homogeneous function of degree zero, we see that  for every $t\in \mathbb F_ q^*$ and $x\in (\mathbb F_q^d, dx),$
$$ \widehat{g}(x)=\sum_{m\in \mathbb F_q^d} \chi(-m\cdot x) g(m/t).$$
From this observation and a change of variables, $m \to tm,$ it follows that

\begin{align*}\|\widehat{g}\|^2_{L^2(S_j, d\sigma)}=&\frac{1}{|S_j|(q-1)} \sum_{t\in \mathbb F_q^*}\sum_{x\in S_j} \left| \sum_{m\in \mathbb F_q^d} \chi(-m\cdot x) g(m/t) \right |^2\\
=&\frac{1}{|S_j|(q-1)} \sum_{t\in \mathbb F_q^*}\sum_{x\in S_j} \left| \sum_{m\in \mathbb F_q^d} \chi(-tm\cdot x) g(m) \right |^2\\
=&\frac{1}{|S_j|(q-1)} \sum_{t\in \mathbb F_q^*}\sum_{\substack{x\in \mathbb F_q^d\\
: x_1^2+\cdots+x_d^2=j}} \left| \sum_{m\in \mathbb F_q^d} \chi(-m\cdot tx) g(m) \right |^2.\end{align*}

Applying a change of variables, $x\to x/t$,  we have
\begin{align*}\|\widehat{g}\|^2_{L^2(S_j, d\sigma)}=&\frac{1}{|S_j|(q-1)} \sum_{\substack{t\in \mathbb F_q^*, x\in \mathbb F_q^d\\: x_1^2+\cdots+x_d^2=jt^2}}\left| \sum_{m\in \mathbb F_q^d} \chi(-m\cdot x) g(m) \right |^2\\
\leq& \frac{1}{|S_j|(q-1)} \sum_{(x, t)\in H_j} \left| \sum_{m\in \mathbb F_q^d} \chi(-m\cdot x) g(m) \right |^2.
\end{align*}

Now, consider a function $G_g:(\mathbb F_q^{d+1}, d\overline{m})\to \mathbb C$ defined by
$$ G_g(m, m_{d+1})= \left\{ \begin{array}{ll}0 \quad &\mbox{if}~~m_{d+1}\neq 0,\\
                                                          g(m) \quad&\mbox{if}~~m_{d+1} =0.\end{array}\right.$$
Then the last expression above can be written by
$$ \frac{1}{|S_j|(q-1)} \sum_{(x, t)\in H_j}  \left| \sum_{(m, m_{d+1})\in \mathbb F_q^d\times \mathbb F_q} \chi(-m\cdot x) \chi(-m_{d+1}\cdot t) G_g(m, m_{d+1})\right|^2.$$
Since $|S_j|(q-1) \sim q^d= |H_j|,$  we see that
$$\|\widehat{g}\|^2_{L^2(S_j, d\sigma_j)} \lesssim \frac{1}{|H_j|} \sum_{(x,t)\in H_j} |\widehat{G_g}(x,t)|^2 =\|\widehat{G_g}\|^2_{L^2(H_j, d\sigma_j )}.$$
Applying  (\ref{res2}) in Lemma \ref{Koh}, we conclude from the definition of $G_g$  that
$$\|\widehat{g}\|^2_{L^2(S_j, d\sigma)}  \lesssim \|G_g\|^2_{L^{(2d+4)/(d+4)}(\mathbb F_q^{d+1}, d\overline{m})} =\|g\|^2_{L^{(2d+4)/(d+4)}(\mathbb F_q^d, dm)},$$
which completes the proof.

\end{document}